\documentclass{amsart}
\usepackage[utf8]{inputenc}
\usepackage{lipsum}
\usepackage{amsmath}
\usepackage{enumerate}
\usepackage{enumitem}
\usepackage{amssymb}
\usepackage{amsthm}
\usepackage{appendix}
\usepackage{hyperref}
\usepackage{bm}
\usepackage{bbm}
\usepackage{graphicx}
\usepackage{color}
\newtheorem{thm}{Theorem}[section]
\newtheorem{cor}[thm]{Corollary}
\newtheorem{lem}[thm]{Lemma}
\newtheorem{prop}[thm]{Proposition}
\newtheorem{defn}[thm]{Definition}
\newtheorem{rem}[thm]{\bf Remark}

\numberwithin{equation}{section}
\newtheorem{sarnak}[thm]{Sarnak conjecture}
\newtheorem{veech}[thm]{Veech's question}
\newtheorem{Chowla}[thm]{Chowla conjecture}
\newtheorem{Taos}[thm]{Logaritmic Sarnak conjecture}
\newtheorem{Taoc}[thm]{Logaritmic Chowla conjecture}
\newtheorem*{Thanks}{\ \ \ \textbf{Acknowledgments}}
\newtheorem*{que}{\ \ \ \textbf{Question}}

\newcommand{\norm}[1]{\left\Vert#1\right\Vert}
\newcommand{\abs}[1]{\left\vert#1\right\vert}
\newcommand{\sand}[1]{\left\{#1\right\}}
\newcommand{\Real}{\mathbb R}
\newcommand{\eps}{\varepsilon}
\newcommand{\To}{\longrightarrow}
\newcommand{\BX}{\mathbf{B}(X)}
\newcommand{\A}{\mathcal{A}}
\newcommand{\bmu}{\bm \mu}
\newcommand{\bmnu}{\bm \nu}
\newcommand{\bml}{\bm \lambda}
\newcommand{\bmo}{\bm O}
\newcommand{\tend}[3][]{\xrightarrow[#2\to#3]{#1}}
\newcommand{\spdj}{\perp_{\mbox{\scriptsize sp}}}
\newcommand{\egdef}{\stackrel{\textrm {def}}{=}}
\newcommand{\ds}{\displaystyle}

\newcommand{\R}{\mathbb{R}}
\newcommand{\E}{\mathbb{E}}

\newcommand{\es}{\mathbb{S}}
\newcommand{\1}{\mathbb{1}}
\renewcommand{\P}{\mathbb{P}}
\newcommand{\Z}{\mathbb{Z}}
\newcommand{\re}{\textrm {Re}}
\newcommand{\N}{\mathbb{N}}
\newcommand{\T}{\mathbb{T}}
\newcommand{\C}{\mathbb{C}}
\newcommand{\D}{\mathbb{D}}
\newcommand{\Ex}{\mathbb{E}}
\newcommand{\Cc}{\mathcal{C}}
\newcommand{\M}{\mathcal{M}}
\newcommand{\U}{\mathcal{U}}
\newcommand{\V}{\mathcal{V}}
\newcommand{\Q}{\mathcal{Q}}
\renewcommand{\S}{\mathcal{S}}
\newcommand{\B}{\mathcal{B}}
\newcommand{\Pin}{\mathcal{P}i}
\newcommand{\I}{\mathcal{I}}
\newcommand{\pr}{\mathcal{P}}
\newcommand{\cl}{\mathcal{L}}
\newcommand{\F}{\mathcal{F}}
\newcommand{\K}{\mathcal{K}}
\newcommand{\pir}{{\mathcal{P}}r}
\newcommand{\qir}{{\mathcal{Q}}r}
\newcommand*{\QEDA}{\hfill\ensuremath{\blacksquare}}
\newcommand{\fr}{\textrm{fr}}

\title{On Veech's proof of Sarnak's theorem on the M\"{o}bius flow}
\author[\MakeLowercase{e.} H. \MakeLowercase{el} Abdalaoui]{\MakeLowercase{el} Houcein \MakeLowercase{el} Abdalaoui}
\date{\today}

\address{Normandie University of Rouen,
	Department of Mathematics, LMRS  UMR 60 85 CNRS\\
	Avenue de l'Universit\'e, BP.12
	76801 Saint Etienne du Rouvray - France .}
\email{elhoucein.elabdalaoui@univ-rouen.fr}



\begin{document}
	\begin{abstract}We present Veech's proof of Sarnak's theorem on the M\"{o}bius flow which say that there is a unique admissible measure on the M\"{o}bius flow. As a consequence, we obtain that Sarnak's conjecture is equivalent to Chowla conjecture with the help of Tao's logarithmic Theorem which assert that the logarithmic Sarnak conjecture is equivalent to logaritmic Chowla conjecture, furthermore, if the even logarithmic Sarnak's conjecture is true then there is a subsequence with logarithmic density one along which Chowla conjecture holds, that is, the M\"{o}bius function is quasi-generic.   
	\end{abstract}

	\maketitle

\section{Introduction}


In this paper, we present Veech's proof of Sarnak's theorem on the  M\"{o}bius flow \cite{VeechNotes}, \cite{VeechNotes2}. Of-course, this proof is connected to Sarnak and Chowla conjectures. Moreover, let us stress that our exposition is self-contained as much as possible. 
\medskip

Roughly speaking, Chowla conjecture assert that the Liouville function is normal, and Sarnak conjecture assert that the M\"{o}bius randomness law holds for any dynamical sequence with zero topological entropy. For more details on the M\"{o}bius randomness law  we refer to \cite{Kowalski}.
\medskip

It is turn out that Veech's proof in combine with the recent result of Tao \cite{BlogTao} yields that Sarnak conjecture implies Chowla conjecture. Indeed, Tao's result assert that if the even logarithmic Chowla conjecture holds then there exists a subsequence $\mathcal{N}$ with logarithmic density $1$ along which the Chowla conjecture holds, and from Veech's proof we will see that this is enough to conclude that Chowla conjecture holds. We recall that T. Tao obtained as a corollary the recent result of Gomilko-Kwietniak-Lema\'{n}czyk  \cite{lem}.
\medskip

Let us further point out that the proof of Gomilko-Kwietniak-Lema\'{n}czyk is based essentially on Tao's theorem on logarithmic Sarnak and Chowla conjectures.
\medskip

We further notice, as T. Tao pointed out, that the proof of Gomilko-Kwietniak-Lema\'{n}czyk  use only that the M\"{o}bius function is bounded.
\medskip

Here, as mentioned before, combining Tao's result with Sarnak's theorem as established by W. Veech, we deduce that Sarnak conjecture holds if and only if Chowla conjecture holds.
\medskip

The more striking result that follows from Veech's proof is the connection between Sarnak conjecture and Hadamard matrix. 
\medskip

We recall that the matrix $H$ of order $n$ is a Hadamard matrix if 
$H$ is a $n \times n$ matrix with entries $\pm 1$ such that 
$HH^{\textrm{T}}=nI_n$, where $I_n$ is the identity matrix. The Hadamard matrix are named after Hadamard since the equality in the famous Hadamard determinant inequality  holds if and only if the matrix is a Hadamard matrix.
\medskip

It is well known that Hadamard matrix exist when $n=1,2$ or $n$ is a multiple of $4$.
\medskip

The Hadamard conjecture states that there is a Hadamard matrix for any multiple of 4. In the opposite direction, the circulant Hadamard matrix conjecture state that the only circulant Hadamard matrix are matrix  of order $1$ and $4$. We recall that a circulant matrix of order $m$ is an $m \times m$ matrix for which
each row except the first is a cyclic permutation of the previous row by one position to the right. 
\medskip

The conjectures of Hadamard are two of the most outstanding unsolved problems in mathematics nowadays.
\medskip

It is well known that the Hadamard matrix is related to the so-called Barker sequences. The Barker sequence is a sequence of $\pm 1$ for which the autocorrelation coefficients are bounded by $1$. Lt us recall that the autocorrelation of a sequence $(x_j)_{j=0}^{N-1}$ are given by
$$c_k =\sum_{j=0}^{N-k-1} x_j x_{j+k},~~~~~~~k \geq 1$$
with
$$c_k =\overline{c_{-k}},~~~~\textrm{if}\; k < 0.$$
For the special real case we have $c_k =c_{-k}$. To be more precise, it is well known that if a Barker sequence of
even length $n$ exists, then so does a circulant Hadamard matrix of order $n$. But, very recently, the author established that there are only finitely many Barker sequences, that is, Turyn-Golay's conjecture is true \cite{elabdal-Erdos}. For more details on the Hadamard matrix, we refer to \cite{Hor}.\\
\begin{Thanks}The author would like to thanks Jean-Paul Thouvenot for the simulating discussion on Sarnak and Chowla conjectures. He is indebted to W. Veech for sending him his notes\footnote{W. Veech in his letter indicated to me that there is only four persons in the world who has a copy of his notes including me. The email of W. Veech is append below.\\
		On 11/03/2016 20:16, Bill Veech wrote:\\
		
		Dear Houcein,\\
		
		Thank you for your message and your lovely papers.
		
		You are right about Lamperti. The fourth edition of Royden
		has deleted some material from the earlier
		editions--I have seen the fourth edition but do not recall if Lamperti
		is still included. (I prefer Royden'ss earlier editions, don while he was still
		alive!)
		
		My class notes are attached. I have no plans for publication
		and am very unlikely to do so. In fact, aside from the students,
		\textbf{I have shared them only with Peter, Jon Fickenscher (a PHD student of mine
			in Princeton), one non-ergodicist friend in  Princeton (who is Godfather to my children)
			and now you}. The notes represent nothing more than one man's attempt to understand for himself a little bit about why Chowla implies Sarnak. 
		
		Thanks again for your papers,
		
		Best regards,
		
		Bill
	} and for many e-simulating discussions on the subject. The author would like also to express his thanks to Mahendra Nadkarni, Mahesh Nerurka and Giovanni Forni for the discussion  or e-discussion on the subject durant the prepartion of this revised version. 
	He would like also to express his thanks to S.G. Dani, Anish Ghosh, and to the organizer of the international conference of algebra and analysis at Pune university, TIFR Mumbia \& CBS of Mumbai university for the invitation. 
\end{Thanks}

\setcounter{section}{2}
\section*{2. Setup and the main result}
The M\"{o}bius function $\bmu$ is related intimately to the Liouville function $\bml$ which is defined by $\bml(n)= 1$ if the number of prime factor of $n$ is even and
$-1$ otherwise. Precisely, the M\"{o}bius function $\bmu$ is given by 
\[
\bmu(n)  = \begin{cases} 
1, \ &\textrm{if } n = 1 \\
\bml(n), \ &\textrm{if } n   \textrm{~is square-free, } \\
0, \ &\textrm{otherwise.}
\end{cases}
\]
We remind that $n$ is square-free if $n$ has no factor in the subset
$\pr_2\egdef\big\{p^2/ p \in \pr\big\}$, where
as customary,  $\pr$ denote the subset of prime numbers.\\

In his seminal paper  \cite{Sar}, P. Sarnak makes the following conjecture.
\begin{sarnak}\label{conj-sarnak} For any dynamical flow $(X,T)$ with topological entropy zero, for any continuous function $f \in C(X)$, for any point $x \in X$, 
	\begin{eqnarray}\label{sarnak-conj}
	\frac{1}{N}\sum_{n=1}^{N}\bmu(n)f(T^nx) \tend{N}{+\infty}0.
	\end{eqnarray}
\end{sarnak}

\noindent{}We recall that the topological entropy of $(X,T)$ is defined by
\[h_{\textrm{top}}(T) = \underset{\varepsilon\to 0}\lim \underset{n\to +\infty}\limsup \dfrac{1}{n}
\textrm{log}~\textrm{sep}(n, T, \varepsilon),\]
where for $n$ integer and $\varepsilon>0$, $\textrm{sep}(n, T, \varepsilon)$ is the maximal possible cardinality of
an ($n, T, \varepsilon$)-separated set in $X$, this later means that for every
two points of it, there exists $0\leq j<n$ with $d(T^j(x), T^j(y)) > \varepsilon$, where $T^{j}$ denotes the $j$-$\textrm{th}$ iterate of $T$. \\

\noindent{}It is well known that an alternative definition can be formulated as follows
\[h_{\textrm{top}}(T) = \underset{\varepsilon\to 0}\lim \underset{n\to +\infty}\limsup \dfrac{1}{n}
\textrm{log}~\textrm{span}(n, T, \varepsilon),\]
where for $n$ integer and $\varepsilon>0$, $\textrm{span}(n, T, \varepsilon)$ is the minimal possible cardinality of
an ($n, T, \varepsilon$)-spanning set in $X$. A set $F$ is an ($n, T, \varepsilon$)-spanning set  if for each point $x \in X$, there exists $y \in F$ such that  $d(T^jx, T^j(y)) \leq \varepsilon$, for $j=0,\cdots,n-1$.\\
We recall further that $\textrm{sep}(n, T, \varepsilon)$  and $\textrm{span}(n, T, \varepsilon)$ increase when $\varepsilon$ decreases. 
We can also use the notion of covering to define the topological entropy as follows
\[h_{\textrm{top}}(T) = \underset{\varepsilon\to 0}\lim \underset{n\to +\infty}\lim \dfrac{1}{n}
\textrm{log}~\textrm{cov}(n, T, \varepsilon),\]
where $\textrm{cov}(n, T, \varepsilon)$ is the least cardinality of a cover of $X$ by open sets $U_1,\cdots ,U_m$ with
$$\sup \Big\{\max_{0\leq j \leq n-1}d(T^jx,T^jy), x,y \in U_i\Big\}<\varepsilon, \; \; i=1,\cdots,m.$$
We thus have 
\begin{align*}
h_{\textrm{top}}(T) &= \underset{\varepsilon\to 0}\lim \underset{n\to +\infty}\limsup \dfrac{1}{n} \textrm{log}~\textrm{span}(n, T, \varepsilon),\\
&=\underset{\varepsilon\to 0}\lim \underset{n\to +\infty}\liminf \dfrac{1}{n}\textrm{log}~\textrm{span}(n, T, \varepsilon)\\
&=\underset{\varepsilon\to 0}\lim \underset{n\to +\infty}\limsup \dfrac{1}{n}
\textrm{log}~\textrm{span}(n, T, \varepsilon)\\
&=\underset{\varepsilon\to 0}\lim \underset{n\to +\infty}\liminf \dfrac{1}{n}
\textrm{log}~\textrm{span}(n, T, \varepsilon)\\
&=\underset{\varepsilon\to 0}\lim \underset{n\to +\infty}\lim \dfrac{1}{n}
\textrm{log}~\textrm{cov}(n, T, \varepsilon).	
\end{align*}

The popular Chowla conjecture on the correlation of the M\"{o}bius function state that 
\begin{Chowla}\label{conj-chowla}
	For any $r\geq 0$, $1\leq a_1<\dots<a_r$, $i_s\in \{1,2\}$ not all equal to $2$, we have
	\begin{equation}\label{cza}
	\sum_{n\leq N}\bmu^{i_0}(n)\bmu^{i_1}(n+a_1)\cdot\ldots\cdot\bmu^{i_r}(n+a_r)={\rm o}(N).
	\end{equation} 
\end{Chowla}
This conjecture is related to the weaker conjecture stated in \cite{chowla}. We refer to \cite{chowla} for more details.
\medskip

In his breakthrough paper \cite{Taol}, T. Tao proposed the following logarithmic version of Sarnak and Chowla conjectures. 

\begin{Taos}\label{Tconj-sarnak} For any dynamical flow $(X,T)$ with topological entropy zero, for any continuous function $f \in C(X)$, for any point $x \in X$, 
	\begin{eqnarray}\label{Tsarnak-conj}
	\frac{1}{\log(N)}\sum_{n=1}^{N}\frac{\bmu(n)f(T^nx)}{n} \tend{N}{+\infty}0.
	\end{eqnarray}
\end{Taos}

The logarithmic Chowla conjecture can be stated as follows:

\begin{Taoc}
	For any $r\geq 0$, $1\leq a_1<\dots<a_r$, $i_s\in \{1,2\}$ not all equal to $2$, we have
	\begin{equation}\label{Tcza}
	\sum_{1 \leq n\leq N}\frac{\bmu^{i_0}(n)\bmu^{i_1}(n+a_1)\cdot\ldots\cdot\bmu^{i_r}(n+a_r)}{n}={\rm o}(\log(N)).
	\end{equation} 
\end{Taoc}
We remind that the logarithmic density of a subset $E \subset \N$ is given by
the following limit (if it exists)

$$\lim_{N \rightarrow \infty} \frac{1}{\log(N)}\sum_{1}^{N} \frac{{\mathbb{I}}_{E}(n)}{n}.$$

Let us further notice that one can replace $\log(N)$ by $\ell_N=\sum_{n=1}^{N}\frac{1}{n}.$ Thanks to Euler estimation.
\medskip

Following L. Mirsky \cite{Mir} and P. Sarnak \cite{Sar}, the subset $A \subset \N$ is admissible if the cardinality $t(p,A)$ of classes modulo $p^2$ in $A$ given by
\[ 
t(p,A) \egdef \left|\bigl\{z\in\Z/p^2\Z: \exists n\in A, n=z\ [p^2]\bigr\}\right| \] satisfy
\begin{equation}
\label{eq:def-admissible}
\forall p \in \pr ,\ t(p,A)<p^2.
\end{equation}
In other words, for every prime $p$ the image of $A$ under reduction mod $p^2$ is proper in $\Z/p^2\Z$.
\medskip

Let $X_3$ be the set $\{0,\pm 1\}^\N$ and $X_2\egdef\{0, 1\}^\N$, and for each $i=1,2$, let $X_i$ be equipped with the product topology. Therefore, $X_3$ and $X_2$ are a compact set.  We denote by $\M_1(X_i)$,  $i=1,2$, the set of the probability measures on $X_i$. it is turn out that $\M_1(X_i)$, $i=1,2$ is a compact set for the weak-star topology by Banach-Alaoglu-Bourbaki  theorem. Let $x \in X_i$, $i=1,2$, and for each $N \in \N$, put 
$$m_N(x)=\frac{1}{N}\sum_{n=0}^{N-1}\delta_{S^nx},$$
where $\delta_y$ is the Dirac measure on $y$ and $S$ is the canonical shift map $(Sx)_n=x_{n+1}$, for each $n \in \N$. Therefore $m_N(x) \in \M_1(X_{i}).$ 
\medskip

We thus get that the weak-star closure  $\I_S(x)$ of the set $\Big\{m_N(x)\Big\} $ is not empty.  We further define the square map $s$ on $X_3$ by
$s(x)=(x_n^2)$ for any $x \in X_3$. 

\begin{defn}
	An infinite sequence $x=(x_n)_{n\in\N^*}\in X_3$ is said to be \emph{admissible} if its \emph{support}
	$\{n\in\N^*: x_n \neq 0\}$ is admissible. In the same way, a finite block $x_1\ldots x_N\in\{0,\pm 1\}^N$ is \emph{admissible} if 
	$\{n\in\{1,\ldots,N\}: x_n \neq 0\}$ is admissible. In the same manner, we define the admissible sets in $X_2$.
\end{defn}
\medskip

For each $i=1,2$, we denote by ${\A_i}$ the set of all admissible sequences in  $X_i$. Since a set is admissible if and only if each of its finite subsets is admissible, and a translation of a admissible set is admissible,   ${\A_i}$ is a closed and shift-invariant subset of $X_i$, \textit{i.e.} a subshift.
We further have that $\bmu^2$ is admissible, and 
$\A_3=s^{-1}(\A_2).$
\medskip

Let us notice that the previous notions has been extended to the so-called $\B$-free setting by el Abdalaoui-Lema\'{n}czyk-de-la-Rue in \cite{elabd-bfree}. Therein , the authors produced a dynamical proof of the Mirsky theorem on the pattern of $\bmu^2$ which assert that the indicator function of the square-free integers is generic for the Mirsky measure $\nu_{M}$, that is, $\bmu^2$ is generic for the push-forward measure of the Haar measure $\mu_h$ of the group $G=\prod_{p}\Z/p^2\Z$ under the map $\varphi~:~G \longrightarrow X_2$ defined by 
\begin{equation}
\label{eq:def_phi} \forall g \in G,~~~
\varphi(g) \egdef \Bigl( f(T^ng)\Bigr)_{n\in\N^*},
\end{equation}
where $T$ is the translation by $\bm1 \egdef (1,1,\cdots)$ and 
$f$ is defined by 
\[
f(g)\egdef \begin{cases}
0 \text{ if there exists $k\ge1$ such that $g_k=0$,}\\
1 \text{ otherwise.}
\end{cases}
\]
We thus get that $\bmu^2=(f(T^n\bmo)),$ $\bmo=(0,0,\cdots)$.
\medskip

We further have that for each measurable subset $C\subset X_2$,
$\nu_{M}(C)\egdef \nu_h(\varphi^{-1}C)$. Then $\nu_{M}=\varphi(\mu_h)$ is shift-invariant, and it can be shown that $\nu_{M}$ is concentrated on ${\A_2}$ and $O(\bmu^2)={\A_2}=\textrm{supp}(\mu_M),$ where $O(\bmu^2) \subset {\A_2}$ is the orbit closure of $\bmu^2$  under the left shift $S$. Moreover, the measurable dynamical system $({\A_2},\nu_{M},S)$ is a factor of $(G,\nu_h,T)$. In particular, it is ergodic, and for any $\eta \in \I_S(\bmu)$, we have
$s\eta=\nu_M$, that is, $\eta(s^{-1}A)=\nu_M(A),$ for any Borel set of ${\A_2}$. Let us notice also that the subset
$\A'=\Big\{\1_A/A \subset \N \textrm{~~finite~~ and~~ admissible} \Big\}$ of ${\A_2}$ is a dense set. For more details, we refer to \cite{elabd-bfree}, \cite{CS}.
\medskip

For any finite sets $A,B \subset \N$, we denote  by $F_{A,B}$ the function
$$F_{A,B}(x)=\Big(\prod_{a \in A}\pi_a(x)\Big) 
\Big(\prod_{b \in B}\pi_b(x)^2\Big),$$
where $\pi_n$ is the $n^{\rm{th}}$ canonical projection given by 
$\pi_n(x)=x_n$, $n \in \N$. Obviously, 
$F_{A,B} \in C(\A_3)$, where $C(\A_3)$ is the space of continuous
function on $X_{A_3}$. We further have 
$F_{A,B}=F_{A,B\setminus (A \cap B)}$, so we can assume always that 
$A$ and $B$ are disjoint.\\

Following W. Veech \cite{VeechNotes}, \cite{VeechNotes2}, we introduce also the notion of admissible measure. 
\begin{defn}A measure $m \in \M_1({\A_3})$ is admissible if 
	\begin{enumerate}[label=(\roman*)]
		\item $Sm=m$, that is, $m(S^{-1}A)=m(A)$, for each Borel set $A \subset {\A_3}$.
		\item $s(m)=\nu_M$, and
		$$\int_{{\A_3}}F_{A,B}(x) dm(x)=0,$$ for any $A \neq \emptyset.$ and $B$ finite sets of $\N$.
	\end{enumerate}
\end{defn}

We are now able to state the main result.
\begin{thm}[Sarnak's theorem on the M\"{o}bius flow \cite{Sar}]\label{main1} There exists a unique admissible measure $\mu_M$. Moreover, $\mu_M$ on ${\A_3}$ is ergodic with the Pinsker algebra
	$$\Pin_{\mu_M}=s^{-1}\Big(\B({\A_2})\Big),$$
	and $\Ex(\pi_1|\Pin_{\mu_M})=0.$	
\end{thm} 
Following W. Veech \cite{VeechNotes2}, the measure $\mu_M$ is called Chowla measure. A direct consequence of Theorem \ref{main1} is the following.

\begin{cor}\label{Veech-Cor}For almost all $x \in {\A_3}$ with respect to $\mu_M$, for any  disjoint sets $A,B$ of $\N$  disjoint sets with 
	$A \neq \emptyset$,
	\begin{eqnarray}
	\lim_{N \rightarrow \infty} \frac{1}{N}\sum_{n=1}^{N}\prod_{a \in A}x(n+a) \prod_{b \in B}x(n+b)^2=0. 
	\end{eqnarray}
	that is, $x$ is generic for $\mu_M$
\end{cor}
\begin{proof} Follows from Birkhoff ergodic theorem.
\end{proof}

\begin{rem}\label{RemarkV}Furthermore, as pointed out by Veech, the existence of the putative ``Chowla measure" does not depend on the Chowla conjecture. We further have that the support of $\mu_M$ is ${\A_3},$  and for any $\eta \in \mathcal{I}_S(\bmu),$
	$\eta({\A_3}^*)=1$, where $\mathcal{I}_S(\bmu)$ is the weak-star closure of the set $\Big\{\frac{1}{N}\sum_{n=0}^{N-1}\delta_{S^n(\bmu)}\Big\},$ and  ${\A_3}^*=\big\{x\in {\A_3}| \textrm{~the~support~}\; A(x) \; \textrm{~is~infinite} \big\}.$ Notice that $S({\A_3}^*)={\A_3}^*$. \\
	
	W. Veech in his lecture notes \cite[p.96, Proof of Remark 24.4]{VeechNotes} established also that if for any $\eta \in \mathcal{I}_S(\bmu)$, we have $\Ex(\pi_1|\Pin_{\eta})=0$, then Sarnak conjecture holds. $\Pin_{\mu}$ stand for the Pinsker algebra of $\eta$. He pointed out that for any $\eta \in \mathcal{I}_S(\bmu)$, we have $\Ex(\pi_1|\Pin_{\mu_M})=0$ and the Chowla conjecture may be seen to be equivalent to the statement that the point $\bmu \in {\A_3}$ is quasi-generic, hence generic for $\mu_M$. This later point will be used later.
\end{rem}

For the proof of our second main result, we need the following  classical result, known as The Portmanteau Theorem \cite{Dudley} (see also \cite{Bil} for its historical facts.)
\begin{lem}\label{PT} Let $(\eta_n)$, $\eta$ be a probability measures on a metric space $(X,d).$ Then, the following are equivalent: 
	\begin{enumerate}[label=(\roman*)]
		\item For any continuous function $f$ on $X$, 
		$$\int f(x)d\eta_n(x) \tend{n}{+\infty} \int f d\eta.$$
		\item \label{open} For any open set $O$, 
		$$\liminf_{n \to \infty}\eta_n(O) \geq \eta(O).$$
		\item For any closed set $F$, 
		$$\limsup_{n \to \infty}\eta_n(F) \leq \eta(F).$$
	\end{enumerate}
\end{lem}

We need also the following result due to T. Tao \cite{BlogTao}.
\begin{thm}[Tao's theorem on logarithmic and non-logarithmic Chowla conjectures\cite{BlogTao}]\label{Taoquasigeneric} 
	Let $k$ be a natural number. Assume that the logarithmically averaged Chowla conjecture is true for $2k$. Then there exists a set ${{\mathcal N}}$ of natural numbers of logarithmic density $1$ such that
	\[
	\lim_{\overset{N \rightarrow \infty} 
		{N \in {\mathcal N}}} \frac{1}{N}\sum_{1}^{N} \lambda(n+h_1) \dots \lambda(n+h_k) = 0,
	\]
	for any distinct ${h_1,\dots,h_k}.$ 	
\end{thm} 
As a corollary, T. Tao obtain the following result which may be found also in\cite{lem}.
\begin{cor}[ $\mu_M$-quasigenercity's theorem ]
	If Sarnak's conjecture holds then there exists a set ${{\mathcal N}}$ of natural numbers such that
	for any $r\geq 0$, $1\leq a_1<\dots<a_r$, $i_s\in \{1,2\}$ not all equal to $2$, we have
	\begin{equation}\label{gkl}
	\frac{1}{N}\sum_{n\leq N}\bmu^{i_0}(n)\bmu^{i_1}(n+a_1)\cdot\ldots\cdot\bmu^{i_r}(n+a_r)
	\tend[{N \in{{\mathcal N}} }]{N}{+\infty}0.
	\end{equation}  
\end{cor}
Combining Sarnak's Theorem \ref{main1} with Tao's Theorem \ref{Taoquasigeneric}, we get the following
\begin{cor}\label{Cor-CS} Sarnak conjecture \ref{conj-sarnak} is equivalent to Chowla conjecture \ref{conj-chowla}.
\end{cor}

\setcounter{section}{3}
\section*{3. Proof of the main result.}
We start by proving the following proposition related to Hadamard matrix. For that,
let $E$ be a finite nonempty set, and $\P(E)$ be the set of subset of $E$. For any $A,B \in \P(E)$, put 
$$C(A,B)=(-1)^{|A \cap B|},$$
where $|.|$ is the cardinality function.  Therefore $C$ is a matrix of order $2^{|E|}$, we further have
\begin{prop}\label{Hadamard}With the notations above, 
	$$\det(C)=\begin{cases} 
	2^{|E|2^{|E|-1}}, \ &\textrm{if $|E|>1$}\\
	-2, \ &\textrm{otherwise}.
	\end{cases}
	$$
	Moreover, if the vector $(\nu(B))_{B \in \P(E)}$ satisfy 
	$$\sum_{B \in \P(E)}C(A,B)\nu(B)= 
	\begin{cases} a, \ &\textrm{if $A=\emptyset$}\\
	0, \ & \textrm{otherwise}.
	\end{cases}
	$$
	Then $$\nu(B)=\frac{a}{2^{|E|}}.$$	
\end{prop}
\begin{proof} The proof of the first part of the proposition can be found in \cite[p.42]{Book}, but for the sake of completeness we include an alternative proof of it.
	\medskip
	
	We start by recalling the Hadamard determinant inequality. Let $M$ be a matrix of order $n$ with real entries and columns $m_1,\cdots,m_n$, then 
	$$\big|\det(M)\big| \leq \prod_{j=1}^{n}\big\|m_j\big\|_2,$$
	where $\|.\|_2$ is the usual Euclidean norm. Therefore, if all the entries are in the interval $[-1,1]$, we get 
	$$\big|\det(M)\big| \leq n^{\frac{n}{2}},$$
	with equality if and only if $M$ is a Hadamard matrix . For short and elementary proof of the Hadamard determinant inequality we refer to \cite [pp.40-41]{Book}, \cite{lang}.\\
	
	We thus need to check that $C$ is a Hadamard matrix. For that, we proceed by induction. For $n=1$, the matrix is given by
	$$C=\begin{pmatrix}
	C(\emptyset,\emptyset) & C(\emptyset,\{1\}) \\
	C(\{1\},\emptyset) & C(\{1\},\{1\})
	\end{pmatrix}=
	\begin{pmatrix}
	1 & 1 \\
	1 & -1
	\end{pmatrix}
	$$
	Assume that the property is true for $n \geq 1$, and let 
	$E_{n+1}=\{1,2,\cdots,n+1\}=E_n \cup \{n+1\}$. We assume that the subsets of 
	$E_{n+1}$ are ordered as those of $E_n$. Notice that this does not affect our proof since the determinant does not depend upon any ordering of the elements of $2^{E}$. It follows that the resulting $2^{n+1}\times 2^{n+1}$ matrix has block form 
	$$C_{n+1}=\begin{pmatrix}
	C_n & C_n \\
	C_n & -C_n
	\end{pmatrix}.$$ We thus get, by Sylvester observation, that $C_{n+1}$ is a Hadamard matrix. For the second part, let $$
	\left(
	\begin{array}{c}
	p\\
	q
	\end{array} \right)\in \R^{2^n} \times \R^{2^n}$$ such that
	$$ C_{n+1} \left(
	\begin{array}{c}
	p\\
	q\\
	\end{array}
	\right)=a. \delta_{\emptyset}(A).$$
	Then, for $n=1$, we have
	$$\begin{cases}
	p+q=a,\ &\textrm{if $A=\emptyset$}\\
	p-q=0,\ & \textrm{if not}.
	\end{cases}$$
	Obviously, we get $p=q=\frac{a}{2}.$ Assume that the property is true for $n$. Then 
	$$p+q=\Big(\frac{a}{2^n},\cdots,\frac{a}{2^n}\Big).$$
	Moreover, since $\det(C_n) \neq 0$, we get $p=q$, that is,
	$$p=q=\Big(\frac{a}{2^{n+1}},\cdots,\frac{a}{2^{n+1}}\Big).$$
	The proof of the lemma is complete.
\end{proof}

For the proof of the Sarnak's theorem \ref{main1}, we need also to characterize the Chowla measure. For that, let us put
$${\qir}_n=\Big\{x \in \{0,1\}^n|  \textrm{~~supp($x$)~is~~admissible}  \Big\},\quad \quad n >0$$  
and 
$$C(x)=\bigcap_{j=1}^{n}\Big\{y \in {\A_2}| \pi_j(x)=x_j\Big\}.$$
Define the partition ${\pir}_n$ by
$${\pir}_n=\Big\{C(x)| x \in \qir_n \Big\}.$$
If follows that any $C(x) \in \pir_n$ admits a partition into $2^{|\textrm{supp}(x)|}$ "cylinder" since 
$$s^{-1}(C(x)) \subset s^{-1}\big({\A_2}\big)={\A_3}.$$ 
More precisely, if $A \subset \textrm{supp}(x)$, then $A$ can be seen as a element $y(A) \in {\A_2}$. We thus denote by $C(x,y(A))$ 
the subset of ${\A_3}$ such that $z \in C(x,y(A))$ if and only if 
the first $n$ coordinates of $z$ are $-1$ on $A$, $1$ on $\textrm{supp}(x)\setminus A$ and $0$ on $[1,n]\setminus \textrm{supp}(x)$. This allows us to see that
$$s^{-1}(C(x))=\bigcup_{A \subseteq \textrm{supp}(x) }C(x,y(A)).$$
Now, for any $A \subset \textrm{supp}(x),$ put
$$G_{A,\textrm{supp}(x)\setminus A}=
F_{A,\textrm{supp}(x)\setminus A} \prod_{c \in [1,n] \setminus \textrm{supp}(x)}(1-\pi_c(y)^2).$$
It is straightforward that $G_{A,\textrm{supp}(x)\setminus A} \in C({\A_3})$. Moreover, 
\begin{eqnarray}\label{expand}
G_{A,\textrm{supp}(x)\setminus A}|s^{-1}\big(C(x)\big)=
F_{A,\textrm{supp}(x)\setminus A}|s^{-1}\big(C(x)\big)
\end{eqnarray}
and  $G_{A,\textrm{supp}(x)\setminus A}$ is identically null on $X_{A_3}\setminus s^{-1}\big(C(x)\big).$\\

Expand the product in the definition of $G_{A,\textrm{supp}(x)\setminus A}$, we get 
$$ G_{A,\textrm{supp}(x)\setminus A} =
\sum_{B \subset [1,n]\setminus \textrm{supp}(x)}
(-1)^{|B|} F_{A,\textrm{supp}(x)\setminus A \cup B}.$$
Now, let $m$ be an admissible measure and assume that $A \neq \emptyset$. Then 
\begin{eqnarray}\label{ve:1}
\int_{{\A_3}} G_{A,\textrm{supp}(x)\setminus A} m(dz)
&=&
\int_{s^{-1}C(x)} F_{A,\textrm{supp}(x)\setminus A}(z) m(dz)\\
&=& 0.
\end{eqnarray}

This gives, for $A \subset \textrm{supp}(x)$ and $A \neq \emptyset$,
$$\sum_{B \subset \textrm{supp}(x)}
(-1)^{|A \cap B|} m(C(x,y(B))=0,$$
since $F_{A,\textrm{supp}(x)\setminus A}$ is constant on each "cylinder" set
$C(x,y(B))$ with the constant value equal to 
$(-1)^{|A \cap B|}.$ \\

We proceed now to evaluate the expression when $A=\emptyset.$ Since 
$sm=\nu_M$, we obtain
$$\nu_M(C(x))=\sum_{B \subset \textrm{supp}(x)} m(C(x,y(B))=0.$$
This combined with Proposition \ref{Hadamard} yields that 
for any $C(x) \in {\pir}_n$, for any $B \subset \textrm{supp}(x)$, we have
\begin{align}\label{equi}
	m(C(x,y(B)))=\frac{\nu_M(C(x))}{2^{|\textrm{supp}(x)|}}.
\end{align}

Summarizing, we conclude that $m$ is completely determined on the partition
${\pir}_n$, i.e., if an admissible measure exists, then it is unique.\\

We proceed now to the proof  of Sarnak 's theorem \ref{main1}.
\medskip

Consider the canonical dynamical system $\big({\A_2} \times \{\pm 1\}^\N, S \times S, \nu_M \otimes m_B(\frac12,\frac12)\big),$ where $m_B(\frac12,\frac12)$ is 
the Bernoulli measure. Therefore, by Furstenberg theorem (Proposition I.3 in \cite{Fur}), the dynamical system $\big({\A_2} \times \{\pm 1\}^\N, S \times S, \nu_M \otimes m_B(\frac12,\frac12)\big),$ is ergodic. We further have, by Theorem 18.13 from \cite[p.325]{Elie}, that the Pinsker algebra satisfies
$$\Pin_{\nu_M \times m_B(\frac12,\frac12)}=
\Pin_{\nu_M} \otimes \Pin_{m_B(\frac12,\frac12)}=
\B({\A_2}) \times \big\{\emptyset, \{\pm 1 \}^\N\big\}$$
up to the null set with respect to $\nu_M \otimes m_B(\frac12,\frac12)$.
\medskip

Now we define a coordinate-wise multiplicative map $\Pi : {\A_2} \times  \{\pm 1 \}^\N \longrightarrow {\A_3}$ by 
$$\Pi(x,\omega)=x.\omega,$$
that is, 
$$\pi_n(\Pi(x,\omega))=x_n \omega_n,\quad \quad n>0.$$
Therefore the dynamical system $({\A_3},S,\mu_M=\Pi(\nu_M \otimes m_B(\frac12,\frac12)))$, where $\mu_M$ is the push-forward measure under $\Pi$, satisfies
\begin{itemize}
	\item $s\mu_M=\nu_M,$ and 
	\item For any $A \neq \emptyset$, we have
	$$\int_{{\A_3}}F_{A,B}(z) \mu_M(dz)=0.$$
\end{itemize}
Whence, $\mu_M$ is admissible, ergodic and $\Pin_{\mu_M}=s^{-1}\B({\A_2})$ up to $\mu_M$ null set. This last fact follows from the following  
$$\Pi^{-1}\Big(\bigcap_{n}S^{-n}\B(\A_3)\Big)\subseteq 
\bigcap_{n}(S \times S)^{-n}\B(\A_2) \times \{\pm 1\}^\N).$$
To finish the proof, we need only to notice that 
$\pi_1=F_{\{1\}, \emptyset}$, and for any finite set $B \subset \N$,
$$\int_{\A_3}\pi_1(y) F_{\emptyset, B}(y) \mu_M(dy)=\int_{\A_3}F_{\{1\},B}(y) \mu_M(dy).$$
We thus conclude that 
$$\Ex(\pi_1|\Pin_{\mu_M})=0,$$ 
up to $\mu_M$ null sets, since the family $\{F_{A,B}\}$ are dense in $C({\A_3}),$ by the classical Stone-Weierstrass theorem.\\

We proceed now to the proof of Corollary \ref{Cor-CS}. For that, we recall the following Veech map. Define
 \begin{align*}
\Phi\colon 
\A_2 \times \{\pm 1\}^{\mathbb{N}} &\longrightarrow \A_3 \\ 
(x,\omega)&\longmapsto (\pi_n(\Phi(x,\omega)))_{n \geq 1},
 \end{align*}
where, for each $n \geq 1$,
\[
\pi_n(\Phi(x,\omega)=\begin{cases}
0, \textrm{~~if~~} n \not \in \textrm{supp}(x),\\
\pi_k(\omega), \textrm{~~if~~} n \in \textrm{supp}(x), n=n_k. 
\end{cases}
\]  
Notice that there is a unique element $(\bmu^2, \omega_0)) \in \A_2 \times \{\pm 1\}^{\mathbb{N}}$ such that $\Phi(\bmu^2, \omega_0)=\bmu.$.
\begin{proof}[\textbf{Proof of Corollary \ref{Cor-CS}.}]The proof of the implication follows from Tao's Theorem \ref{Taoquasigeneric} and since the admissible measure is unique. Indeed, Sarnak's conjecture implies Chowla logarithmic conjecture along a subsequence of full density logarithmic. Therefore, by Tao's theorem \ref{Taoquasigeneric} combined with the admissibility of the measure, we get that
that for $n \geq 1$, for any $C(x) \in {\pir}_n$ and for any $B \subset \textrm{supp}(x)$, 
\begin{align}
\lim_{k \rightarrow +\infty}\frac{1}{N_k}\Big|\Big\{1 \leq m \leq N_k~:~ S^{m}(\bmu) \in C(x,y(B))\Big\}\Big|=\frac{\nu_M(C(x))}{2^{|\textrm{supp}(x)|}},
\end{align}
by \eqref{equi}. Notice that $\bmu(m+l) \in \textrm{supp}(x)$ if and only if $m+l$ is square-free and under Veech map, $\pi_k(\omega_0) \in \textrm{supp}(x)$, with $m+l=n_k$. We further get that	
	 $O(\bmu)={\A_3}=\textrm{supp}(\mu_M),$ where $O(\bmu) \subset {\A_3}$ is the orbit closure of $\bmu$ 
	under the left shift $S$. Indeed, by \ref{open} in Lemma \ref{PT}, we have,
	\begin{align}
	\liminf_{k \rightarrow +\infty} \Big(\frac{1}{N_k} \sum_{m=1}^{N_k} \delta_{S^m(\bmu)}(O(\bmu)^c)\Big) \geq
	\mu_M(O(\bmu)^c),
	\end{align}	
	where $O(\bmu)^c$ is the complement of $O(\bmu)^c$ which is open. We thus obtain 
	$$\mu_M(O(\bmu))=1.$$ 
\noindent{} This combined with Corollary \ref{Veech-Cor}. yields that for any block of order $n$, the frequencies of $-1$ and $1$ in $\bmu$ are relatively uniformly distributed , that is, the second component of the image of $\bmu$ under Veech map is a normal number. We thus get that Chowla conjecture holds (see also the deep and nice survey \cite[Section 1.]{Bill2}\footnote{Precisely, subsections 1.8 to 1.11 where a generic points for the subshift is constructed by applying Parathasarathy theorem.} which is related to the Remark \ref{RemarkV}).\\ For the converse, there are several proofs by Sarnak \cite{lett}, Tao \cite{Taol}, Veech \cite{VeechNotes2},  and el Abdalaoui-Kua\l ga-Przymus-Lema\'{n}czyk-de la Rue\cite{elabdal-chowla}. 
	
	 
\end{proof}
\begin{rem}We emphasize that the proof in \cite{elabdal-chowla} is given in the abstract setting. Therein, the authors considered the abstract arithmetical sequences $(u_n)$ and under the assumption  that the square $(u_n^2)$ is generic for the so-called abstract Mirsky measure, they established that there is a relatively Bernouilli extension to the space $\{-1,1,0\}^{\Z}$. But, in this abstract setting, there is no Sarnak-Veech's theorem (Theorem \ref{main1}) valuable and it may fail for a large class of arithmetical sequences. This is due to the lack of the arithmetical structure connected to prime numbers. We further stress that the property of multiplicativity is not enough. Therefore, one need to be careful about the existence of Chowla measure which is defined only on the admissible subshift. For the M\"{o}bius function, the existence of it is is guaranteed by  Sarnak-Veech's theorem (Theorem \ref{main1}). This later approach was pushed forward by e. H. el Abdalaoui and M. Nerurkar in \cite{AN}. Therein, the author present a dissection of M\"{o}bius flow based on the relatively Kolomogorov extension tools due to Rhoklin-Sinai and exploit the main theorem presented here. It is turn out that recently, some authors used the ideas of \cite{AN} in the so-called abstract setting but as pointed out in this remark there is no Sarnak-Veech's theorem valuable for the abstract setting and the statement may fail for an appropriate choice of a sequence. Let us point out also that in \cite{AN}, the authors obtain a dynamical proof of some number theoretical results without speed of convergence. Finally, we notice that  Veech's conjecture and Sarnak conjecture are equivalent according to the main results of this paper. We recall that Veech conjecture is formulated as a question in his one of the last preprints  \cite{Veech-May} as follows. 
\begin{veech}For any $\eta \in \mathcal{I}_S(\bmu)$, is it true
$$	\E(\phi_1|{\Pin_{\eta}})=0 ~~~~\eta-\textrm{almost~~everywhere?}
$$\end{veech}
\noindent where $\Pin_{\eta}$ stand for the Pinsker algebra associated to $\eta$. He further proved that if the answer is affirmative then Sarnak's conjecture holds.  Obviously, by Sarnak-Veech's theorem (Theorem \ref{main1}), Chowla conjecture implies that the answer is affirmative  
\end{rem}

\begin{que}\label{QV} Let us recall that $O(\bmu) \subset {\A_3}$ is the orbit closure of $\bmu$ 
	under the left shift $S$. Do we have that $O(\bmu)={\A_3}$?
\end{que}

\section*{Declarations}
\paragraph{\textbf{Conflict of interest.}} The author declare that he have no conflict of interest.



\begin{thebibliography}{9}
	\bibitem{AN}
	e. H. el Abdalaoui, and M. Nerurkar, \emph{Sarnak's M\"{o}bius disjointness for dynamical systems with singular spectrum and dissection of M\"{o}bius flow.}, preprint, 2020,  	
	arXiv:2006.07646 [math.DS]
	\bibitem{elabd-bfree}  e. H. El Abdalaoui, M. Lema\'{n}czyk, and T. de la Rue, A dynamical point of view on the set of B-free integers, Int. Math. Res. Not. IMRN 2015, no.16, 7258-7286. 
	
	\bibitem{elabdal-Erdos}E. H. el Abdalaoui, On the Erd\"{o}s flat polynomials problem, Chowla conjecture and Riemann Hypothesis, arXiv:1609.03435 [math.CO]
	
	\bibitem{elabdal-chowla}
	e. H. el Abdalaoui, J. Kua\l ga-Przymus; M. Lema\'{n}czyk; T. de la Rue,
	The Chowla and the Sarnak conjectures from ergodic theory point of view,
	Discrete Contin. Dyn. Syst. 37 (2017), no. 6, 2899-2944.  
	\bibitem{Book}
	M. Aigner and G. M. Ziegler, Proofs from The Book. Fifth edition. Including illustrations by Karl H. Hofmann. Springer-Verlag, Berlin, 2014. 
	\bibitem{Bil}
	P. Billingsley,  \emph{Probability and measure,} anniversary edition with a foreword by Steve Lalley and a brief biography of Billingsley by Steve Koppes. Wiley Series in Probability and Statistics. John Wiley \& Sons, Inc., Hoboken, NJ, 2012. 
	\bibitem{CS}
	F. Cellarosi, Ya.G.~Sinai, \emph{Ergodic Properties of Square-Free Numbers,} J. Eur. Math. Soc. (JEMS) 15 (2013), no. 4, 1343-1374.
	\bibitem{chowla}
	S.~Chowla, \emph{The Riemann hypothesis and Hilbert's tenth problem}, Mathematics and Its Applications, Vol 4, Gordon and Breach Science Publishers, New York, 1965.
	\bibitem{Dudley}
	R. M. Dudley, \emph{Real analysis and probability,} Revised reprint of the 1989 original, Cambridge Studies in Advanced Mathematics, 74. Cambridge University Press, Cambridge, 2002. 
	\bibitem{Fur}
	H. Furstenberg, Disjointness in ergodic theory, minimal sets, and a problem in Diophantine approximation, Math. Systems Theory 1 1967 1-49.
	\bibitem{Elie}
	E. Glasner, \emph{Ergodic theory via joinings,} Mathematical Surveys and Monographs, 101, American Mathematical Society, Providence, RI, 2003.
	\bibitem{lem}
	A. Gomilko, D. Kwietniak, M. Lema\'{n}czyk, Sarnak's conjecture implies the Chowla conjecture along a subsequence, in: Ergodic theory and dynamical systems in their
	interactions with arithmetics and combinatorics, 237-247, Lecture Notes in Math.,
	2213, Springer, Cham, 2018. arXiv:1710.07049 [math.NT]
	\bibitem{Hor}
	K. J. Horadam, Hadamard matrices and their applications. Princeton University Press, Princeton, NJ, 2007. xiv+263 pp. 
	\bibitem{Kowalski}H.~Iwaniec and E. Kowalski, \emph{Analytic number theory,} American Mathematical Society Colloquium Publications, 53, American Mathematical Society, Providence, RI, 2004.
	\bibitem{lang}
	K. Lange, Hadamard's determinant inequality, Amer. Math. Monthly 121 (2014), no. 3, 258-259. 
	\bibitem{Mir}
	L. Mirsky, Arithmetical pattern problems relating to divisibility by $r$-th powers, Proc. London Math. Soc. (2) 50, (1949). 497-508.
	\bibitem{Sar}
	P.~Sarnak.
	\newblock Three lectures on the {M}{\"o}bius function, randomness and dynamics.
	\newblock{http://publications.ias.edu/sarnak/}.
	\bibitem{lett}
	P. Sarnak, private letter.
	
	\bibitem{BlogTao}T. Tao, {https://terrytao.wordpress.com/2017/10/20/the-logarithmically-averaged-and-non-logarithmically-averaged-chowla-conjectures/}
	
	\bibitem{Taol}
	T. Tao, Equivalence of the logarithmically averaged Chowla and Sarnak conjectures. Number theory, Diophantine problems, uniform distribution and applications, 391-421, Springer, Cham, 2017. 
	\bibitem{Bill2}
	W. Veech, Topological Dynamics, Bull. Amer. Math. Soc., Vol 83, N. 5, 775-829.  
	\bibitem{VeechNotes}W. Veech, A conjecture Between the Chowla and Sarnak conjectures, Lecture notes.
	\bibitem{Veech-May}
	W. Veech, A Conjecture Between the Chowla and
	Sarnak Conjectures, preprint, May 6, 2016.( private communication.)
	\bibitem{VeechNotes2} W. Veech, M\"{o}bius dynamics, Lectures Notes, Spring Semester 2016, +164 pp, private communication.
	
\end{thebibliography}
\end{document}